\documentclass[10pt,reqno]{amsart}
\usepackage{bbm}
\usepackage{mathrsfs}
\usepackage{amsfonts} 
\usepackage[dvipsnames,usenames]{color}
\usepackage{graphicx,latexsym,bm,amsmath,amssymb,verbatim,multicol,lscape}
\newtheorem{thm}{Theorem} [section]
\newtheorem{lem}{Lemma}[section]

\theoremstyle{definition}

\theoremstyle{remark}

\numberwithin{equation}{section}

\allowdisplaybreaks

\begin{document}
\title{Sums of polynomial-type exceptional units modulo $n$}
\author[J.Y. Zhao]{Junyong Zhao}
\address{Mathematical College, Sichuan University, Chengdu 610064, P.R. China}
\address{School of Mathematics and Statistics, Nanyang Institute of Technology,
Nanyang 473004, P.R. China}
\email{zhjy626@163.com (J.Y. Zhao)} 
\author[S.F. Hong]{Shaofang Hong$^{*}$}
\address{Mathematical College, Sichuan University, Chengdu 610064, P.R. China}
\email{sfhong@scu.edu.cn; s-f.hong@tom.com; hongsf02@yahoo.com (S.F. Hong)}
\author[C.X. Zhu]{Chaoxi Zhu}
\address{Mathematical College, Sichuan University, Chengdu 610064, P.R. China}
\address{Science and Technology on Communication Security Laboratory,
Chengdu 610041, P.R. China}
\email{zhuxi0824@126.com (C.X. Zhu)}
\begin{abstract}
Let $f(x)\in\mathbb{Z}[x]$ be a nonconstant polynomial. Let $n, k$
and $c$ be integers such that $n\ge 1$ and $k\ge 2$. An integer $a$
is called an $f$-exunit in the ring $\mathbb{Z}_n$ of residue
classes modulo $n$ if $\gcd(f(a),n)=1$. In this paper, we use the
principle of cross-classification to derive an explicit formula
for the number ${\mathcal N}_{k,f,c}(n)$ of solutions $(x_1,...,x_k)$
of the congruence $x_1+...+x_k\equiv c\pmod n$ with all $x_i$ being
$f$-exunits in the ring $\mathbb{Z}_n$. This extends a recent result
of Anand {\it et al.} [On a question of $f$-exunits in
$\mathbb{Z}/{n\mathbb{Z}}$, {\it Arch. Math. (Basel)} {\bf 116} (2021),
403-409]. We derive a more explicit formula for
${\mathcal N}_{k,f,c}(n)$ when $f(x)$ is linear or quadratic.
\end{abstract}
\thanks{$^*$S.F. Hong is the corresponding author and was supported
partially by National Science Foundation of China Grant \#11771304.}
\keywords{Polynomial-type exceptional unit, exponential sum,
ring of residue classes, principle of cross-classification}
\subjclass[]{11B13, 11L03, 11L05}
\maketitle

\section{Introduction}
Let $\mathbb{Z}$ and $\mathbb{Z}^+$ stand for the set of integers
and the set of positive integers, respectively. For any
$n\in\mathbb{Z}^+$, we set $\mathbb{Z}_n\!=\!\{0, 1, ..., n-1\}$
to be the ring of residue classes modulo $n$. Let
$\mathbb{Z}_n^*=\{s\in\mathbb{Z}_n: \gcd(s, n)=1\}$ be the group
of the units in $\mathbb{Z}_n$. Throughout this paper,
let $k$ be an integer with $k\ge 2$. In 1925, Rademacher \cite{[R]}
asked for an explicit formula for the number $N(k,c,n)$ of solutions
$(x_1,...,x_k)\in(\mathbb{Z}_n^*)^{k}$ of the linear congruence
$$
x_1+...+x_k\equiv c\pmod n.
$$
In 1926, Brauer \cite{[B]} solved this problem by showing that
\begin{align} \label{thmR}
N(k,c,n)=\frac{\varphi^k(n)}{n}\Big(\prod_{p|n,p|c}\Big(1-\frac{(-1)^{k-1}}
{(p-1)^{{k-1}}}\Big)\Big)\Big(\prod_{p|n,p\nmid c}\Big(1-\frac{(-1)^k}{(p-1)^{k}}\Big)\Big),
\end{align}
where $\varphi(n)$ is Euler's totient function and the products
are taken over all prime divisors $p$ of $n$. In 2009, Sander \cite{[SJ1]}
gave a new proof of the formula for $N(2,c,n)$ by using
the multiplicativity of $N(2,c,n)$ with respect to $n$.

The concept of exceptional units was first given by Nagell \cite{[NT]}
in 1969; he introduced it to solve certain cubic Diophantine equations.
For any commutative ring $R$ with the identity elment $1_R$, let $R^*$
denote the multiplicative group of units in $R$. An element $a\in R$
is said to be an {\it exceptional unit} if both $a\in R^*$ and
$1_R-a\in R^*$. Many types of Diophantine equations including
Thue equations \cite{[TN1]}, Thue-Mahler equations \cite{[TN2]},
discriminant form equations \cite{[SN2]} and others have been
studied by means of exceptional units (for more references, see
\cite{[Lou]}). On the other hand, with the help of exceptional units,
Lenstra \cite{[LH]} introduced a new method to find Euclidean
number fields. Exceptional units also have connections with cyclic
resultants (\cite{[SC]} and \cite{[SC2]}) and Lehmer's conjecture
related to Mahler measure (\cite{[Si]} and \cite{[Si2]}).

Following Sander's notation in \cite{[SJ]}, we use the coinage
{\it exunit} to stand for exceptional unit. As usual, for any
integer $m$ and prime number $p$, we let $\nu_p(m)$ stand for
the {\it $p$-adic valuation} of $m$, that is, $\nu_p(m)$ is
the unique nonnegative integer $r$ such that $p^r|m$ and
$p^{r+1}\nmid m$. We denote by $\omega(m):=\sum_{p \ {\rm prime}, \ p|m}1$
the number of distinct prime divisors of $m$.
Yang and Zhao \cite{[YZ]} extended Sander's result \cite{[SJ]}
by showing that the number of ways to represent each element of
$\mathbb{Z}_n$ as the sum of $k$ exceptional units is given by
$$
(-1)^{k\omega(n)}\prod_{p| n}p^{k\nu_p(n)-\nu_p(n)-k}
\Big(p\sum_{j=0\atop{j\equiv c\pmod p}}^k\binom{k}{j}+(2-p)^k-2^k\Big).
$$
(This corrects an error in the formula in Theorem 1 of \cite{[YZ]},
where the sign factor $(-1)^k$ should read $(-1)^{k\omega(n)}$.)

We can easily observe that for any $a\in \mathbb{Z}^+$, $a\pmod n$
is an exunit in $\mathbb{Z}_n$ if and only if $\gcd(a(1-a),n)=1$.
In other words, $a\pmod n$ is an exunit in $\mathbb{Z}_n$ if and
only if $\gcd(f(a), n)=1$ with $f(x)=x(1-x)$. This observation
naturally motivates the concept of $f$-exunit as follows. Let
$n\ge 1$ be an integer and let $f(x)\in\mathbb{Z}[x]$.
An integer $a$ is an {\it $f$-exunit} in the ring $\mathbb{Z}_n$
if $\gcd(f(a), n)=1$ (see, for example, \cite{[ACR1]}). We denote
by $E_{f}(n)$ the set of all $f$-exunits in the ring $\mathbb{Z}_n$.
It is clear that
$$
E_{f}(n)=\{a\in\mathbb{Z}_n: \gcd(f(a),n)=1\}.
$$
Throughout, we assume that $f$ is a
nonconstant polynomial and $c$ is an integer.
For any finite set $S$, we denote by $\sharp S$
the number of the elements in $S$. We set
${\mathcal N}_{k,f,c}(n)$ to be the number of
solutions $(x_1,...,x_k)$ of the congruence
$x_1+...+x_k\equiv c\pmod n$ with
$x_1,...,x_k\in E_{f}(n)$, that is,
$$
{\mathcal N}_{k,f,c}(n):=\sharp\{(x_1,...,x_k)\in E_{f}(n)^k
: x_1+...+x_k\equiv c \pmod n\}.
$$
For any given prime number $p$, associated with $f(x)\in\mathbb{Z}[x]$
and $c\in\mathbb{Z}$, we define the nonnegative number
${\mathcal M}_{k,f,c}(p)$ by
\begin{small}
\begin{align}\label{1}
{\mathcal M}_{k,f,c}(p):= \sharp\{ ( x_1 ,..., x_{ k - 1 })\in\mathbb Z_p^{k - 1}:
f(x_1) \cdots f( x_{ k - 1 })f(c-\sum_{i=1}^{k-1}x_i)\equiv  0 \pmod p\}.
\end{align}
\end{small}

\noindent In \cite{[ACR]}, Anand {\it et al.} presented
a formula for ${\mathcal N}_{2,f,c}(n)$ which also extends
Sander's theorem \cite{[SJ]}. However, it still remains open
to give an explicit formula for ${\mathcal N}_{k,f,c}(n)$
when $k\ge 3$.

In this paper, we introduce a new method to investigate
the number ${\mathcal N}_{k,f,c}(n)$. Actually, we make
use of the well-known principle of cross-classification
\cite{[Ap]} to derive an explicit formula of
${\mathcal N}_{k,f,c}(n)$ for all positive integers $n$.
The first main result of this paper can be stated as follows.
\begin{thm}\label{thm1}
Let $f(x)\in\mathbb{Z}[x]$ be a nonconstant polynomial
and let $c$ be an integer. Then ${\mathcal N}_{k,f,c}$
is a multiplicative function, and for any positive
integer $n$, we have
$$
{\mathcal N}_{k,f,c}(n)=n^{k-1}\prod_{p| n}
\Big(1-\frac{{\mathcal M}_{k,f,c}(p)}{p^{k-1}}\Big).
$$
\end{thm}

Theorem \ref{thm1} reduces to the main result of
\cite{[ACR]} (Theorem 2) if $k=2$. For a linear
polynomial $f(x)$, we have the second main result
of this paper generalising (\ref{thmR}).
\begin{thm}\label{thm2}
Let $c$ and $n$ be integers with $n\ge 1$. Let
$f(x)=ax+b\in\mathbb{Z}[x]$ with $\gcd(a,n)=1$.
Then
$${\mathcal N}_{k,f,c}(n)=n^{k-1}\prod_{p|n}
\frac{(p-1)^k+(-1)^k\delta_p}{p^k},$$
where
\begin{align}\label{eq1.2}
\delta_p:=
\begin{cases}
p-1,\ $~${\it if}\ p|(ac+kb),\\
-1,\ ~$~\!$~$~$~{\it if}\ p\nmid(ac+kb).
\end{cases}
\end{align}
\end{thm}

For a quadratic polynomial $f(x)$, one can deduce from
Theorem 1.1 the third main result of this paper, which
extends the Yang-Zhao theorem \cite{[YZ]}.
\begin{thm}\label{thm3}
Let $c$ and $n$ be integers with $n\ge 1$. Let
$f(x)=(a_1x-a_2)(b_1x-b_2)$ with $a_1,a_2,b_1,b_2\in \mathbb Z$
and $\gcd(a_1,n)=\gcd(b_1,n)=\gcd(a_1b_2-a_2b_1, n)=1$. Then
$$
{\mathcal N}_{k,f,c}(n)=(-1)^{k\omega(n)}\prod_{p| n}
p^{k\nu_p(n)-\nu_p(n)-k}\Big(p\sum_{j=0\atop{(a_2b_1-a_1b_2)j
\equiv a_1b_1c-a_1b_2k\!\!\pmod p}}^k\binom{k}{j}+(2-p)^k-2^k\Big).
$$
\end{thm}

This paper is organized as follows. In Section 2 we present
several lemmas that are needed in the proofs of Theorems
\ref{thm1} and \ref{thm3}. Section 3 is devoted to the
proof of Theorem 1.1 and Sections 4 and 5 to the proofs
of Theorems \ref{thm2} and \ref{thm3}, respectively.

\section{Preliminary lemmas}

In this section, we supply several lemmas that will be needed
in the proofs of Theorems \ref{thm1} and \ref{thm3}. We begin with the celebrated
principle of cross-classification.

\begin{lem} {\rm (Principle of cross-classification)}
{\rm [3, Theorem 5.31]}
\label{lem1} Let $R$ be any given finite set. For a subset $T$ of
$R$, we denote by $\bar T$ the set of those elements of $R$
which are not in $T$. If $R_1,...,R_{m-1}$
and $R_m$ are arbitrary $m$ given distinct subsets of $R$, then
$$\sharp\bigcap_{i=1}^m\bar R_i=\sharp R+\sum_{t=1}^m(-1)^t
\sum_{1\leq i_1<...<i_t\leq m}\sharp\bigcap_{j=1}^tR_{i_j}.$$
\end{lem}

The next result can be proved by using the Chinese remainder theorem.

\begin{lem} \label{lem2} {\rm [3, Theorem 5.28]}
Let $r\in \mathbb Z^+$ and $g(x_1,...,x_r)\in\mathbb{Z}
[x_1,...,x_r]$, and let $m_1,..., m_k$ be pairwise
relatively prime positive integers. For any integer
$i$ with $1\le i\le k$, let $N_i$ be the number of
zeros of $g(x_1,...,x_r)\equiv 0\pmod {m_i}$ and let $N$
denote the number of zeros of $g(x_1,...,x_r)\equiv 0
\pmod {\prod_{i=1}^k m_i}$. Then $N=\prod_{i=1}^k N_i$.
\end{lem}

\begin{lem}\label{lem3}
Let $n,r, m\in \mathbb Z^+$ with $m|n$ and let
$g(x_1,...,x_r)\in \mathbb Z[x_1,...,x_r]$. Then
\begin{align*}
&\sharp\{(x_1,...,x_r)\in \mathbb Z_n^r: g(x_1,...,x_r)\equiv 0 \pmod m\}\\
=&\Big(\frac{n}{m}\Big)^r\sharp\{(x_1,...,x_r)\in \mathbb Z_m^r:
g(x_1,...,x_r)\equiv 0 \pmod m\}.
\end{align*}
\end{lem}
\begin{proof}
Let $(x_1, ..., x_r)$ be any $r$-tuple of integers
with all $x_i$ in the set $\{0,1,...,m-1\}$ such that
$$
g(x_1,...,x_r)\equiv 0 \pmod m.
$$
Then, for arbitrary integers $i_1,...,i_r$ with
$0\leq i_1,...,i_r\leq \frac{n}{m}-1$,
$$0\le x_1+i_1m,...,x_r+i_rm\le n-1$$
and
$$g(x_1+i_1m,...,x_r+i_rm)\equiv 0 \pmod m.$$
Every such $i_j \ (1\le j\le r)$ has $\frac{n}{m}$
choices. So, one can immediately deduce the assertion
in the lemma.
\end{proof}

In the following, let $p$ be a prime number. For any
integer $a$ coprime to $p$, let $a^{-1}$ stand for an
integer satisfying that $aa^{-1}\equiv 1\pmod p$.
For any $a_1, a_2, b_1, b_2\in\mathbb{Z}$
with $\gcd(a_1a_2,p)=1$, let $f_1(x)=(a_1x-b_1)(a_2x-b_2)$
and $f_2(x)=(x-a_1^{-1}b_1)(x-a_2^{-1}b_2)$. Then
$f_1(t)=a_1a_2f_2(t)$ for any integer $t$. It follows that
\begin{align*}
E_{f_1}(p)=\{t\in\mathbb{Z}_p: \gcd(f_1(t),p)=1\}
=\{t\in\mathbb{Z}_p: \gcd(f_2(t),p)=1\}=E_{f_2}(p).
\end{align*}
So, for our purpose, if $f(x)$ is a reducible quadratic
polynomial with no multiple zeros, then we can assume that
$f(x)=(x-a)(x-b)\in \mathbb Z[x]$. Let us now compute
${\mathcal M}_{k,f,c}(p)$.

\begin{lem}\label{lem4}
Let $a,b\in \mathbb Z_p$ with $a\not=b$ and let $f(x)=(x-a)(x-b)$.
Then
\begin{small}
\begin{align*}
&{\mathcal M}_{k,f,c}(p)=p^{k-1}-\frac{(-1)^k}{p}
\Big(p\sum_{j=0\atop{(a-b)j\equiv c-bk\pmod p}}^k\binom{k}{j}+(2-p)^k-2^k\Big).
\end{align*}
\end{small}
\end{lem}
\begin{proof}
Since $f(x)=(x-a)(x-b)\in\mathbb{Z}[x]$,
\begin{align*}
&{\mathcal M}_{k,f,c}(p)=\sharp\{(x_1,...,x_{k-1})\in
\mathbb Z_p^{k-1} : f(x_1)\cdots f(x_{k-1})
f(c-\sum_{i=1}^{k-1}x_i)\equiv 0 \pmod p\}\\
=&\sharp\mathbb Z_p^{k - 1}-\sharp\{(x_1,...,x_{k - 1})
\in\mathbb Z_p^{k-1}: f(x_1)\cdots f(x_{k - 1})f(c -\sum_{i=1}^{k-1}x_i)
\not\equiv  0 \pmod p\}\\
=&p^{k-1}-\sharp\{(x_1, ..., x_{k - 1 })\in \mathbb Z_p^{k-1}:
(c-\sum_{i=1}^{k-1}x_i-a)(c-\sum_{i=1}^{ k - 1 }x_i - b)
\prod_{i=1}^{k - 1}(x_i-a)(x_i - b)\not\equiv 0\!\!\!\!\!\! \pmod p\}\\
=&p^{k-1}-\sharp\{(x_1,...,x_{k})\in \big(\mathbb Z_p\backslash
\{a,b\}\big)^k: \sum_{i=1}^{k}x_i\equiv c\pmod p\}.
\end{align*}
Moreover,
\begin{align*}
&\sharp\{(x_1,...,x_{k})\in \big(\mathbb Z_p\backslash\{a,b\}\big)^k:
\sum_{i=1}^{k}x_i\equiv c\pmod p\}\\
=&\frac{1}{p}\sum _{y=0}^{p-1}\sum _{(x_1,...,x_k)\in(\mathbb Z_p\backslash\{a,b\})^k}
\exp\Big(\frac{2\pi {\rm i}y(x_1+...+x_k-c)}{p}\Big)\\
=&\frac{1}{p}\sum _{y=0}^{p-1}\Big(\sum _{x\in\mathbb Z_p\backslash\{a,b\}}
\exp\Big(\frac{2\pi {\rm i}xy}{p}\Big)\Big)^k\exp\Big(\frac{2\pi {\rm i}(-cy)}{p}\Big)\\
=&\frac{1}{p}\Big(\sum _{y=1}^{p-1}\Big(-\exp\Big(\frac{2\pi {\rm i}ay}{p}\Big)
-\exp\Big(\frac{2\pi {\rm i}b y}{p}\Big)\Big)^k\exp\Big(\frac{2\pi {\rm i}(-cy)}{p}\Big)+(p-2)^k\Big)\\
=&\frac{1}{p}\Big((-1)^k\sum _{y=1}^{p-1}\Big(\sum_{j=0}^{k}\binom{k}{j}\exp\Big(\frac{2\pi {\rm i}y(ja+(k-j)b))} {p}\Big)\exp\Big(\frac{2\pi {\rm i}(-cy)}{p}\Big)+(p-2)^k\Big)\\
=&\frac{1}{p}\Big((-1)^k\sum _{j=0}^{k}\binom{k}{j}\sum_{y=1}^{p-1}
\exp\Big(\frac{2\pi {\rm i}y(aj+bk-bj-c)}{p}\Big)+(p-2)^k\Big)\\%
=&\frac{(-1)^k}{p}\Big(\sum_{j=0\atop{(a-b)j\equiv c-bk\pmod p}}^{k}\binom{k}{j}(p-1)
-\sum_{j=0\atop{(a-b)j\not\equiv c-bk\pmod p}}^{k}\binom{k}{j}+(2-p)^k \Big)\\
=&\frac{(-1)^k}{p}\Big(p\sum_{j=0\atop{(a-b)j\equiv c-bk\pmod p}}^k\binom{k}{j}+(2-p)^k-2^k\Big).
\end{align*}
We then deduce that
\begin{align*}
&{\mathcal M}_{k,f,c}(p)=p^{k-1}-\frac{(-1)^k}{p}
\Big(p\sum_{j=0\atop{(a-b)j\equiv c-bk\pmod p}}^k\binom{k}{j}+(2-p)^k-2^k\Big)
\end{align*}
as desired. Lemma 2.4 is proved.
\end{proof}

\section{Proof of Theorem \ref{thm1}}

In this section, we use Lemmas \ref{lem1}, \ref{lem2} and \ref{lem3}
to show Theorem \ref{thm1}.\\
\\
{\it Proof of Theorem \ref{thm1}.}
First of all,
\begin{align}\label{3.1}
&{\mathcal N}_{k,f,c}(n)=\sharp\{(x_1,...,x_{k})\in E_{f}(n)^k:
x_1+...+x_k\equiv c \pmod n\}\\ \notag
&=\sharp\{(x_1 ,..., x_{k})\in \mathbb Z_n^k:
x_1 + ... + x_k \equiv c \pmod n, \gcd(f(x_1),n)
=...=\gcd(f(x_k), n)=1\}\\ \notag
&=\sharp\{(x_1 ,..., x_{k-1})\in\mathbb Z_n^{k-1}:
\gcd(f( x_1 ), n)=... =\gcd(f(x_{k-1}), n )
=\gcd(f(c-\sum_{i=1}^{k-1}x_i),n)=1\}.
\end{align}

Let $n=p_1^{r_1}\cdots p_s^{r_s}$ be the standard prime
factorisation of $n$. In Lemma \ref{lem1}, let
$R=\mathbb Z_n^{k\!-\!1}$ and, for any integer $i$
with $1\le i\le s$, let
$$
R_i=\{(x_1,...,x_{k-1})\in R: f(x_1)\cdots f(x_{k-1})
f(c-\sum_{i=1}^{k-1}x_i)\equiv0\pmod {p_i}\}.
$$
Then
\begin{align}\label{3.2}
\bar R_i& =\{(x_1,...,x_{k-1})\in R: f(x_1)\cdots f(x_{k-1})
f(c-\sum_{i=1}^{k-1}x_i)\not\equiv0 \pmod {p_i}\}\notag\\
&=\!\{(\!x_1\!,\!...,\!x_{\!k\!-\!1}\!)\in R\!:
\gcd(\!f(\!x_1\!)\!,p_i\!)\!=\!...\!=\!\gcd(\!f(\!x_{k\!-\!1}\!)\!,p_i\!)
\!=\!\gcd(f(c-\sum_{i=1}^{k-1}x_i),p_i)=1\}.
\end{align}
It follows from (\ref{3.1}) and (\ref{3.2}) that
\begin{align}\label{3.3}
{\mathcal N}_{k,f,c}(n)=\sharp\bigcap_{i=1}^{s}\bar R_i,
\end{align}
and, for arbitrary integers $i_1,..., i_t$ with $1\le i_1<...<i_t\le s$,
\begin{align*}
\bigcap_{j=1}^tR_{i_j}
=\{(x_1,...,x_{k-1})\in R: f(x_1)\cdots f(x_{k-1})f(c-\sum_{i=1}^{k-1}x_i)
\equiv0 \pmod {\prod_{j=1}^{t} p_{i_j}}\}.
\end{align*}

On the other hand, by Lemmas \ref{lem2} and \ref{lem3},
\begin{align}\label{3.4}
&\sharp\bigcap_{j=1}^tR_{i_j}=
\sharp\{(x_1,...,x_{k-1})\in R: f(x_1)\cdots f(x_{k-1})f(c-\sum_{i=1}^{k-1}x_i)
\equiv 0\pmod {\prod_{j=1}^{t} p_{i_j}}\}\notag\\
=&\Big(\frac{n}{\prod_{j=1}^{t}p_{i_j}}\Big)^{k-1}
\sharp\{( x_1,..., x_{ k - 1 })\in\mathbb Z_{\prod_{j=1}^tp_{i_j}}^{ k - 1}:\notag\\
& \ \ \ \ \ \ \ \ \ \ \ \ \ \ \ \ \ \ \ \  \ \ \ \ \ \ \ \ \ \ \  \ 
f(x_1) \cdots  f(x_{ k - 1 })f(c-\sum_{i=1}^{k-1}x_i)
\equiv0\pmod {\prod_{j=1}^{t} p_{i_j} } \}\notag\\
=& n^{k - 1} \prod_{j=1}^{t}\frac{1}{p_{i_j}^{k-1}}
 \sharp\{ ( x_1,..., x_{ k - 1 })\in\mathbb Z_{p_{i_j}}^{ k - 1}:
 f(x_1) \cdots f( x_{k - 1})f(c-\sum_{i=1}^{k-1}x_i)\equiv 0
\pmod{p_{i_j}}\}\notag\\
=&n^{k-1}\prod_{j=1}^{t}\frac{{\mathcal M}_{k,f,c}(p_{i_j})}{p_{i_j}^{k-1}}.
\end{align}
It then follows from Lemma \ref{lem1}, (\ref{3.3}) and (\ref{3.4}) that
\begin{align*}
{\mathcal N}_{k,f,c}(n)=\sharp\bigcap_{i=1}^{s}\bar R_i
=&\sharp R+\sum_{t=1}^s(-1)^t\sum_{1\leq i_1<...<i_t\leq s}
\sharp\bigcap_{j=1}^tR_{i_j}\\
=&n^{k-1}+\sum_{t=1}^s(-1)^t\sum_{1\leq i_1<...<i_t\leq s}n^{k-1}\prod_{j=1}^{t}\frac{{\mathcal M}_{k,f,c}(p_{i_j})}{p_{i_j}^{k-1}}\\
=&n^{k-1}\Big(1+\sum_{t=1}^s(-1)^t\sum_{1\leq i_1<...<i_t\leq s}\prod_{j=1}^{t}\frac{{\mathcal M}_{k,f,c}(p_{i_j})}{p_{i_j}^{k-1}}\Big)\\
=&n^{k-1}\prod_{p|n}\Big(1-\frac{{\mathcal M}_{k,f,c}(p)}{{p}^{k-1}}\Big)
\end{align*}
as required. This concludes the proof of Theorem \ref{thm1}. \hfill$\square$

\section{Proof of Theorem \ref{thm2}}
In this section, we present the proof of Theorem \ref{thm2}.\\
\\
{\it Proof of Theorem \ref{thm2}.} Choose a prime $p$ with $p|n$,
so that $\gcd(a,p)=1$. By (\ref{1}),
\begin{align*}
&{\mathcal M}_{k,f,c}(p)
=\sharp\{ ( x_1,...,x_{k-1})\in\mathbb Z_p^{k-1}: (a(c-\sum_{i=1}^{k-1}x_{i})+b)
\prod_{i=1}^{k-1}(ax_i+b) \equiv  0 \pmod p\}\\
=&\sharp\{ ( x_1,...,x_{k-1})\in\mathbb Z_p^{ k - 1}:
(c-\sum_{i=1}^{k-1}(x_{i}+ba^{-1})+kba^{-1})\prod_{i=1}^{k-1}
(x_i+ba^{-1})\equiv 0 \!\!\pmod p\}.
\end{align*}
Letting $y_i=x_i+ba^{-1}$ for $1\leq i\leq k-1$ gives
\begin{align}\label{326eq1}
{\mathcal M}_{k,f,c}(p)
=&\sharp\{ ( y_1,...,y_{k-1})\in\mathbb Z_p^{ k - 1}: y_1 \cdots y_{k-1}
(c+kba^{-1}-\sum_{i=1}^{k-1}y_{i})\equiv 0\pmod p\}\notag\\
=&\sharp\mathbb Z_p^{ k-1}-\sharp\{(y_1,...,y_{k-1})\in\mathbb Z_p^{ k - 1}:
y_1 \cdots y_{k-1}( c+kba^{-1}-\sum_{i=1}^{k-1}y_{i})\not\equiv  0 \pmod p\}\notag\\
=&\sharp\mathbb Z_p^{ k-1}-\sharp(\mathbb Z_p^*)^{k-1}+
\sharp\{(y_1,...,y_{k-1})\in(\mathbb Z_p^*)^{k-1}: c+kba^{-1}-\sum_{i=1}^{k-1}y_{i}
\equiv  0 \pmod p\}\notag\\
=&p^{k-1}-(p-1)^{k-1}+\sharp\{ ( y_1,...,y_{k-1})\in(\mathbb Z_p^*)^{k-1}:
\sum_{i=1}^{ k-1}y_{i} \equiv  c+kba^{-1} \pmod p\}\notag\\
=&p^{k-1}-(p-1)^{k-1}+N(k-1,c+kba^{-1},p).
\end{align}
But (\ref{thmR}) tells us that
\begin{align}\label{326eqN}
N(k-1,c+kba^{-1},p)
=\begin{cases}
\frac{(p-1)^{k-1}+(-1)^{k-1}(p-1)}{p},\ {\rm if}\ p|(c+kba^{-1}),\\
\frac{(p-1)^{k-1}+(-1)^{k}}{p},\ {\rm if}\ p\nmid (c+kba^{-1}).
\end{cases}
\end{align}
Combining (\ref{326eqN}) with (\ref{326eq1}), 
\begin{align}\label{eq42}
{\mathcal M}_{k,f,c}(p)=p^{k-1}-\frac{(p-1)^k+(-1)^k\delta_p}{p},
\end{align}
with $\delta_p$ being given as in (\ref{eq1.2}).
Finally, by Theorem \ref{thm1} and (\ref{eq42}), 
\begin{align}\label{eq41}
{\mathcal N}_{k,f,c}(n)=& n^{k-1}\prod_{p|n}
\Big(1-\frac{{\mathcal M}_{k,f,c}(p)}{p^{k-1}}\Big)=n^{k-1}\prod_{p|n}
\frac{(p-1)^k+(-1)^k\delta_p}{p^k}.\notag
\end{align}
This finishes the proof of Theorem \ref{thm2}. \hfill$\Box$

\section{Proof of Theorem \ref{thm3}}

In this section, we use Theorem \ref{thm1} and Lemma \ref{lem4}
to prove Theorem \ref{thm3}. \\
\\
{\it Proof of Theorem \ref{thm3}}.
Choose a prime $p$ with $p|n$ and let $h(x)=(x-a)(x-b)$
with $a=a_2a_1^{-1}$ and $b=b_2b_1^{-1}$ taken modulo $p$.
The inverses exist and $a\not\equiv b\pmod p$ because
$\gcd(a_1, p)=\gcd(b_1, p)=\gcd(a_1b_2-a_2b_1, p)=1$ and
for $t\in\mathbb{Z}$, $f(t)\equiv 0\pmod p$ if and only
if $h(t)\equiv 0\pmod p$. Applying Lemma 2.4 to $h(x)$ gives
\begin{align*}
{\mathcal M}_{k,f,c}(p)
=&{\mathcal M}_{k,h,c}(p)\\
=&p^{k-1}-\frac{(-1)^k}{p}\Big(p\sum_{j=0\atop{(a-b)j\equiv c-bk
\pmod p}}^k\binom{k}{j}+(2-p)^k-2^k\Big)\\
=&p^{k-1}-\frac{(-1)^k}{p}\Big(p\sum_{j=0\atop{(a_2a_1^{-1}-b_2b_1^{-1})j
\equiv c-b_2b_1^{-1}k \pmod p}}^k\binom{k}{j}+(2-p)^k-2^k\Big).
\end{align*}
Then, applying Theorem \ref{thm1},
\begin{align*}
&{\mathcal N}_{k,f,c}(n)
=n^{k-1}\prod_{p| n}\Big(1-\frac{{\mathcal M}_{k,f,c}(p)}{p^{k-1}}\Big)\\
=&n^{k-1}\prod_{p|n}\frac{(-1)^k}{p^{k}}\Big(p\sum_{j=0
\atop{(a_2a_1^{-1}-b_2b_1^{-1})j\equiv c-b_2b_1^{-1}k\!\!\pmod p}}^k\binom{k}{j}+(2-p)^k-2^k\Big)\\
=&(-1)^{k\omega(n)}\prod_{p|n}p^{k\nu_p(n)-\nu_p(n)-k}\Big(p\sum_{j=0
\atop{(a_2a_1^{-1}-b_2b_1^{-1})j\equiv c-b_2b_1^{-1}k\!\!\pmod p}}^k\binom{k}{j}+(2-p)^k-2^k\Big)\\
=&(-1)^{k\omega(n)}\prod_{p|n}p^{k\nu_p(n)-\nu_p(n)-k}\Big(p\sum_{j=0
\atop{(a_2b_1-a_1b_2)j\equiv a_1b_1c-a_1b_2k\!\!\pmod p}}^k\binom{k}{j}
+(2-p)^k-2^k\Big)
\end{align*}
as expected. This completes the proof of Theorem \ref{thm3}. \hfill$\Box$

\begin{center}
{\sc Acknowledgements}
\end{center}
The authors would like to thank Dr. Yulu Feng and
the anonymous referee for their careful reading of
the paper and very helpful suggestions and comments
that improve the presentation of the paper.

\bibliographystyle{amsplain}

\begin{thebibliography}{}
\bibitem{[ACR1]} Anand, J. Chattopadhyay and B. Roy,
On sums of polynomial-type exceptional units in
$\mathbb{Z}/{n\mathbb{Z}}$, Arch. Math. (Basel)
{\bf 114} (2020), 271-283.
\bibitem{[ACR]} Anand, J. Chattopadhyay and B. Roy,
On a question of $f$-exunits in $\mathbb{Z}/{n\mathbb{Z}}$,
Arch. Math. (Basel) {\bf 116} (2021), 403-409.
\bibitem{[Ap]} T.M. Apostol, Introduction to analytic
number theory, Springer-Verlag, New York, 1976.
\bibitem{[B]} A. Brauer, Losung der Aufgabe 30, Jahresber.
Dtsch. Math.-Ver. {\bf 35} (1926), 92-94.
\bibitem{[LH]} H.W. Lenstra, Euclidean number fields of large degree,
Invent. Math. {\bf 38} (1976/1977), 237-254.
\bibitem{[Lou]} S.R. Louboutin, Non-Galois cubic number fields
with exceptional units, Publ. Math. Debr. {\bf 91} (2017), 153-170.
\bibitem {[NT]} T. Nagell, Sur un type particulier d'unites algebriques,
Ark. Mat. {\bf 8} (1969), 163-184.
\bibitem{[R]} H. Rademacher, Aufgabe 30, Jahresber. Dtsch. Math.-Ver.
{\bf 34} (1925) 158.
\bibitem{[SJ1]} J.W. Sander, On the addition of units and nonunits
mod $m$, J. Number Theory {\bf 129} (2009), 2260-2266.
\bibitem{[SJ]} J.W. Sander, Sums of exceptional units in residue
class rings, J. Number Theory {\bf 159} (2016), 1-6.
\bibitem{[Si]} J.H. Silverman, Exceptional units and numbers of
small Mahler measure, Exp. Math. {\bf 4} (1995), 69-83.
\bibitem{[Si2]} J.H. Silverman, Small Salem numbers, exceptional
units, and Lehmer's conjecture, Rocky Mt. J. Math. {\bf 26}
(1996), 1099-1114.
\bibitem{[SN2]} N.P. Smart, Solving discriminant form equations
via unit equations, J. Symbolic Comput. {\bf 21} (1996) 367-374.
\bibitem{[SC]} C.L. Stewart, Exceptional units and cyclic resultants,
Acta Arith. {\bf 155} (2012), 407-418.
\bibitem{[SC2]} C.L. Stewart, Exceptional units and cyclic
resultants, II. Diophantine methods, lattices, and
arithmetic theory of quadratic forms, 191-200, Contemp. Math.,
587, Amer. Math. Soc., Providence, RI, 2013.
\bibitem{[TN1]} N. Tzanakis and B.M.M. deWeger, On the practical
solution of the Thue equation, J. Number Theory {\bf 31} (1989), 99-132.
\bibitem{[TN2]} N. Tzanakis and B.M.M. deWeger, How to explicitly
solve a Thue-Mahler equation, Compos. Math. {\bf 84} (1992), 223-288.
\bibitem{[YZ]} Q.H. Yang and Q.Q. Zhao, On the sumsets of exceptional
units in $\mathbb{Z}_n$, Monatsh. Math. {\bf 182} (2017), 489-493.
\end{thebibliography}

\end{document}